\documentclass{amsart}
\usepackage{verbatim}
\usepackage{rotating} %for \includegraphix[angle=270]
\usepackage{amssymb}
\newtheorem{theorem}{Theorem}[section]
\newtheorem{thm}{Theorem}[section]
\newtheorem{proposition}[theorem]{Proposition}
\newtheorem{prop}[theorem]{Proposition}

\newtheorem{cor}[theorem]{Corollary}
\newtheorem{question}[theorem]{Question}
\newtheorem{definition}[theorem]{Definition}

\newtheorem{conjecture}[theorem]{Conjecture}

\theoremstyle{plain}
\numberwithin{equation}{theorem}

\theoremstyle{remark}

\newcommand{\C}{{\mathbb C}}

\newcommand{\Q}{{\mathbb Q}}

\newcommand{\Z}{{\mathbb Z}}

\newcommand{\cX}{{\mathcal X}}
\newcommand{\cY}{{\mathcal Y}}
\newcommand{\cXb}{{\overline \cX}}

\newcommand{\cV}{{\mathcal V}}

\newcommand{\OO}{{\mathcal O}}

\newcommand{\fo}{\mathfrak o}

\newcommand{\fm}{\mathfrak m}
\newcommand{\fp}{\mathfrak p}

\newcommand{\ba}{ {\bar \alpha} }
\newcommand{\bL} {\overline L }

\newcommand{\vb}{\vec{\beta}}
\newcommand{\Kbar}{\overline K}

\DeclareMathOperator{\Per}{Per}

\newcommand{\Qbar}{\bar{\Q}}

\DeclareMathOperator{\Spec}{Spec}

\DeclareMathOperator{\GL}{GL}

\DeclareMathOperator{\Aut}{Aut}

\DeclareMathOperator{\tor}{tor}

\DeclareMathOperator{\bN}{\mathbb{N}}

\newcommand{\bP}{{\mathbb P}}
\newcommand{\bZ}{{\mathbb Z}}
\newcommand{\bG}{{\mathbb G}}
\newcommand{\Zp}{{\mathbb{Z}_p}}

\newcommand{\bA}{{\mathbb A}} 
\newcommand{\bQ}{{\mathbb Q}}

\newcommand{\lra}{\longrightarrow}
\newcommand{\cO}{\mathcal{O}}
\newcommand{\cF}{\mathcal{F}}

\newcommand{\cU}{\mathcal{U}}
\newcommand{\cG}{\mathcal{G}}
\newcommand{\cH}{\mathcal{H}}
\newcommand{\cL}{\mathcal{L}}

\begin{document}

\title[Bounding periods of subvarieties]{Applications of $p$-adic analysis for bounding periods of subvarieties under \'etale maps}

\author{J.~P.~Bell}
\address{
Jason Bell\\
Department of Pure Mathematics\\
University of Waterloo\\
Waterloo, ON N2L 3G1\\
CANADA 
}
\email{jpbell@uwaterloo.ca}

\author{D.~Ghioca}
\address{
Dragos Ghioca\\
Department of Mathematics\\
University of British Columbia\\
Vancouver, BC V6T 1Z2\\
Canada
}
\email{dghioca@math.ubc.ca}

\author{T.~J.~Tucker}
\address{
Thomas Tucker\\
Department of Mathematics\\
University of Rochester\\
Rochester, NY 14627\\
USA
}
\email{ttucker@math.rochester.edu}

\begin{abstract}
Using methods of $p$-adic analysis we give a different proof of Burnside's problem for automorphisms of quasiprojective varieties $X$ defined over a field of characteristic $0$. More precisely, we show that any finitely generated torsion subgroup of ${\rm Aut}(X)$ is finite. In particular this yields effective bounds for the size of torsion of any semiabelian variety over a finitely generated field of characteristic $0$. More generally, we obtain effective bounds for the length of the orbit of a preperiodic subvariety $Y\subset X$ under the action of an \'etale endomorphism of $X$.
\end{abstract}

\maketitle

\section{Introduction}

In \cite{MorSil1}, Morton and Silverman conjecture that there is a
constant $C(N,d,D)$ such that for any morphism $f: \bP^N \lra \bP^N$
of degree $d$ defined over a number field $K$ with $[K:\bQ] \leq D$,
the number of preperiodic points of $f$ over $K$ is less than or equal
to $C(N,d,D)$.  This conjecture remains very much open, but in the
case where $f$ has good reduction at a prime $\fp$, a great deal has
been proved about bounds depending on $\fp$, $N$, $d$, $D$ (see
\cite{mike-thesis, Pezda, Hutz-0}).  

In this paper, we study the more general problem of bounding periods of
subvarieties of any dimension.  We prove the following.

\begin{theorem}
\label{key result for subvarieties}
Let $K$ be a finite extension of $\Q_p$, let $\fo_v$ be the ring of integers of $K$, let $k_v$ be its residue field and let $e$ be the ramification index of $K/\Q_p$. 
Let $\cX$ be a smooth $\fo_v$-scheme whose generic
fiber $X$ has dimension $g$, let $\Phi: \cX \lra \cX$ be \'etale, 
let $\cY$ be a subvariety of $\cX$, and assume there is a point on $\cY(\fo_v)$ which is smooth on the generic fiber of $\cY$. If $\cY$
is preperiodic under the action of $\Phi$, then the length of its
orbit is bounded
by $p^{1+r}\cdot \#\GL_g(k_v)\cdot \#\overline{\cX}(k_v)$, where  $\overline{\cX}$ is the special fiber of $\cX$, and $r$ is the smallest nonnegative integer larger than $(\log(e)-\log(p-1))/\log(2)$. 
\end{theorem}

Theorem~\ref{key result for subvarieties} is proved using $p$-adic
analytic parametrization of forward orbits under the action of an
\'etale endomorphism of a quasiprojective variety.  The same method can be
used to study finitely generated torsion subgroups of
$\Aut(X)$, when $X$ is a quasi-projective variety defined over a field
$K$ of characteristic zero.  Theorem \ref{the key result} gives an upper bound on the size
of the largest finitely generated torsion subgroup in $\Aut X$  when $X$ has a smooth
model over a finite extension of the $p$-adic integers; the bound
depends only on the dimension of $X$ and on the number of points in
the special fiber of this model.  This gives rise to a new proof of the following
theorem of Bass and Lubotzky \cite{BL}.
\begin{theorem}
\label{Burnside p}
Let $X$ be a geometrically irreducible quasiprojective variety defined over a field of characteristic $0$. Then each finitely generated torsion subgroup $H$ of $\Aut(X)$  is finite.
\end{theorem}

Theorem \ref{Burnside p} shows in particular that the Burnside problem has an affirmative solution for automorphism groups of quasiprojective varieties.  We recall that the Burnside problem is said to have a positive answer for a group $G$ if every finitely generated torsion subgroup of $G$ is finite.  The first substantial result in this area was due to Burnside (c.f. \cite[\S 9]{Lam2001}), who showed that if $H$ is a (not necessarily finitely generated) torsion subgroup of ${\rm GL}_n(\mathbb{C})$ of exponent $d$ then the order of $H$ could be bounded in terms of $d$ and $n$.  Using a specialization argument and applying Burnside's result, Schur \cite{Schur1911} later showed that every finitely generated torsion subgroup of ${\rm GL}_n(\mathbb{C})$ is finite.   Proofs of geometric Burnside-type results generally proceed along similar lines as that of the Burnside-Schur theorem: one first uses specialization to reduce to the case that the base field is a finitely generated extension of the prime field; one then shows that in this case a torsion subgroup necessarily has bounded exponent and is finite.  An interesting problem that arises naturally is to then bound the exponent and the order of a torsion subgroup of ${\rm Aut}(X)$ in terms of geometric data and the field $k$ of definition for $X$ when $k$ is a finitely generated extension of $\mathbb{Q}$.  Some work in this direction has been done by Serre \cite{Serre2009}, who gave sharp upper-bounds on the sizes of torsion subgroups of the group of birational transformations of $\mathbb{P}^2(k)$ when $k$ is a finitely generated extension of $\mathbb{Q}$.  
 
The $p$-adic analytic parametrization of forward orbits under the
action of an automorphism of a quasiprojective variety can be used in
different directions as well.  We can prove the following result.
\begin{theorem}
\label{Zhang's conjecture for automorphisms}
Let $X$ be an irreducible quasiprojective variety defined over $\Qbar$ of dimension larger than $1$, and let $\Phi$ be an automorphism of $X$ for which there exists no  nonconstant $f\in \Qbar(X)$ such that $f\circ \Phi = f$. Then there exists a codimension-$2$ subvariety $Y$ whose orbit under $\Phi$ is Zariski dense in $X$.
\end{theorem}
For any subvariety $Y$, its orbit under $\Phi$ is the union of all $\Phi^n(Y)$ for $n\in\bN$.

Theorem~\ref{Zhang's conjecture for automorphisms} yields positive evidence to two conjectures in arithmetic geometry. On one hand, we have the \emph{potential density question}, i.e. describe the class of varieties $X$ defined over $\Qbar$ for which there exists a number field $K$ such that $X(K)$ is Zariski dense in $X$ (see \cite{Amerik, Bogomolov-Tschinkel, Harris-Tschinkel} for various
results regarding potentially dense varieties). Our Theorem~\ref{Zhang's conjecture for automorphisms} yields that if $Y$ is potentially dense, then $X$ is potentially dense.  On the other hand, if $X$ is a surface, then Theorem~\ref{Zhang's conjecture for automorphisms} yields the existence of a point $x\in X(\Qbar)$ whose orbit is Zariski dense in $X$ (note that in this case, $Y$ is a finite collection of points and since $X$ is irreducible, we obtain that a single orbit under $\Phi$ must be Zariski dense in $X$). Hence, Theorem~\ref{Zhang's conjecture for automorphisms} yields a positive answer for automorphisms of surfaces for the following conjecture (proposed independently by Amerik,
Bogomolov and Rovinsky \cite{Amerik}, and Medvedev and Scanlon
\cite{polynomials}).
\begin{conjecture}
\label{new conjecture}
Let $X$ be a quasiprojective variety defined over an algebraically closed field $K$ of characteristic $0$. Let $\Phi:X\lra X$ be an endomorphism defined over $K$ such that there exists no positive dimensional variety $Y$ and no dominant rational map $\Psi:X\lra Y$ such that $\Psi\circ \Phi = \Psi$ generically. Then there exists $x\in X(K)$ such that $\OO_\Phi(x)$ is Zariski dense in $X$.
\end{conjecture}
Alternatively one can formulate the hypothesis in Conjecture~\ref{new conjecture}  by asking that $\Phi$ does not preserve a rational fibration, i.e. there exists no nonconstant $f\in K(X)$ such that $f\circ \Phi = f$. It is immediate to see that if $\Phi$ preserves a rational fibration, then there is no point with Zariski dense orbit under $\Phi$.  Conjecture~\ref{new conjecture} strengthens a conjecture of Zhang, which was also the motivation for both Amerik, Bogomolov and Rovinsky, respectively for Medvedev and Scanlon for formulating the above Conjecture~\ref{new conjecture}. Motivated by the dynamics of endomorphisms on abelian varieties,  Zhang \cite{ZhangLec} proposed the
following question: given a projective variety $X$ defined over a
number field $K$ endowed with a polarizable endomorphism $\Phi$, is
there a point $x\in X(\Kbar)$ whose orbit under $\Phi$ is Zariski
dense in $X$?  We say that
$\Phi$ is polarizable if there exists an ample line bundle $\cL$ on
$X$ such that $\Phi^*(\cL)=\cL^{\otimes d}$ (in ${\rm Pic}(X)$) for some integer
$d>1$.

In \cite{Amerik-Campana}, Amerik and Campana prove 
Conjecture~\ref{new conjecture} for projective varieties of trivial canonical bundle
defined over an uncountable field $K$.  However, if $K=\Qbar$, then
the problem seems much more difficult. Only recently, Medvedev and Scanlon \cite{polynomials} proved Conjecture~\ref{new conjecture} when $\Phi=(f_1,\dots, f_N)$ is an endomorphism of $\bA^N$ given by $N$ one-variable polynomials $f_i$ defined over $\Qbar$. Also, Junyi \cite[Theorem 1.4]{xie_junyi_preperiodic} proved Conjecture~\ref{new conjecture} for birational maps on projective surfaces. Finally, connected to Conjecture~\ref{new conjecture}, we mention Amerik's result \cite{Amerik2} who proved (using the $p$-adic approach introduced in \cite{BGT}) that most
orbits of algebraic points are infinite under the action of an arbitrary
rational self-map (of infinite order).

Our proof of Theorem~\ref{Zhang's conjecture for automorphisms} uses a result (see Theorem~\ref{thm: bound}) which gives an upper bound for the period of codimension-$1$ subvarieties of $X$ which are periodic under $\Phi$; in particular this yields that the union of all periodic hypersurfaces is Zariski closed.  We note that Cantat \cite[Theorem A]{Cantat_preperiodic} proved a similar bound for the number of periodic hypersurfaces under the stronger hypothesis that there exist no nonconstant rational function $f$ and no constant $\alpha$ such that $f\circ \Phi=\alpha\cdot f$. 
Our Theorem~\ref{key result for subvarieties} yields that each periodic subvariety with a point over some complete $v$-adic field has bounded period. So, if $Y$ is a codimension-$2$ subvariety of $X$ which is neither periodic, nor contained in one of the finitely many codimension-$1$ periodic subvarieties, then its orbit under $\Phi$ is Zariski dense. Using the same approach, it is immediate to get the existence of codimension-$1$ subvarieties with a Zariski dense orbit in $X$. 

Using the hypothesis that $X$ contains a Zariski dense orbit, and also using  Vojta's proof of the Mordell-Lang Theorem for semiabelian varieties (see \cite{V1}) we obtain the following stronger bound for the number and period of codimension-$1$ periodic subvarieties.
\begin{theorem}
\label{Vojta consequence}
  Let $X$ be a quasi-projective variety and let 
  $\sigma: X \lra X$ be an automorphism defined over a number
  field $K$ and suppose that there is a a point $x \in X(K)$ such that
  the orbit of $x$ under $\sigma$ is Zariski dense in $X$.  Then any
  $\sigma$-invariant closed subset $W$ of $X$ has at most $\dim X - 
  h^1(Y,\cO_Y) + \rho(Y)$ geometric components of codimension one, where $Y$ is a projective closure of $X$ and 
  $\rho(Y)$ is the Picard number of $Y$.
\end{theorem}

Of course, Conjecture~\ref{new conjecture} for an automorphism $\sigma: X
\lra X$ follows immediately whenever one knows that the union of all
$\sigma$-invariant subvarieties is Zariski closed.  Hence, the
following equivalence is of interest here.  

\begin{definition} 
\label{DM definition}
Let $X$ be a quasi-projective variety over a field $K$ and let $\sigma:X\to X$ be an automorphism of $X$.  We say that $(X,\sigma)$ satisfies the \emph{geometric Dixmier-Moeglin equivalence} if the following are equivalent for each $\sigma$-stable subvariety $Y$ of $X$:
\begin{enumerate}
\item[$(1)$] there exists a point $y\in Y$ such that $\{\sigma^n(y)\colon n\in \mathbb{Z}\}$ is Zariski dense in $Y$;
\item[$(2)$] the union of all proper $\sigma$-invariant subvarieties of $Y$ is Zariski closed; 
\item[$(3)$] there does not exist a non-trivial $f\in k(Y)$ such that $f\circ \sigma=f$.  
\end{enumerate}
\end{definition}
We note that the geometric Dixmier-Moeglin equivalence does not hold in general---for example, there are H\'enon maps of $\mathbb{A}^2$ with the property that $(3)$ holds but $(2)$ does not (c.f. Devaney and Nitecki \cite{DeNit} and Bedford and Smillie \cite[Theorem 1]{BedSm})---but it is conjectured to hold when $X$ is smooth and projective and $\sigma$ has zero entropy.   As before, for $X$ a complex variety, we have the implications $(2)\implies (1)\implies (3)$ \cite{BRS}. Theorem~\ref{Zhang's conjecture for automorphisms} proves that the equivalences from Definition~\ref{DM definition} hold for any surface.

Here is the plan of our paper: in Section~\ref{v-adic analysis} we prove some preliminary results (see Proposition~\ref{Mike's bound}) for rigid analytic functions, which we use then in Section~\ref{Burnside section} for proving Theorems~\ref{key result for subvarieties} and \ref{Burnside p} and their corollaries. In Section~\ref{bounds for hypersurfaces} we find an upper bound for the period of codimension-$1$ periodic subvarieties under the action of an automorphism $\Phi$ of a quasiprojective variety which does not preserve a rational fibration (see Theorem~\ref{thm: bound}). In Section~\ref{Zhang surfaces}, using Theorem~\ref{thm: bound}, we prove   
Theorem~\ref{Zhang's conjecture for automorphisms}  and  Theorem~\ref{Vojta consequence}.  Finally, we conclude with Section~\ref{other questions} in which we discuss related questions (in the spirit of Poonen's conjectures \cite{Poonen-uniform}) about uniform boundedness for periods of points in algebraic families of endomorphisms. 

\medskip

\emph{Acknowledgments.} We thank Serge Cantat and Zinovy Reichstein
for useful discussions regarding Burnside's problem.  We thank
Ekaterina Amerik, Najmuddin Fakhruddin, and Xie Junyi for helpful
comments and corrections on an earlier version of this paper.

%%%%%%%%%%%%%%%%%%%%%%%%%%%%%%%%%%%%%%%%%%%%%%%%%%%%%%%%%%%%%%%%%%%%%%%%%%%%
%%%%%%%%%%%%%%%%%%%%%%%%%%%%%%%%%%%%%%%%%%%%%%%%%%%%%%%%%%%%%%%%%%%%%%%%%%%%%

\section{Nonarchimedean analysis}
\label{v-adic analysis}

\subsection{Power series}

The setup for this section is as follows: $p$ is a prime number, $K_v/\Q_p$ is a finite extension, while the $v$-adic norm $|\cdot |_v$ satisfies $|p|_v=1/p=|p|_p^{1/e}$ (i.e., $e$ is the ramification index for this extension). We let $\fo_v$ be the ring of $v$-adic integers of $K_v$, let $\pi$ be a uniformizer of $\fo_v$, and we let $k_v$ be its residue field. 

We let $g$ be a positive integer, and let $c$ be a positive real number. For two power series $F, G\in \fo_v[[z_1,\dots, z_g]]$, we write $F \equiv G\pmod{p^c}$ if each coefficient $a_\alpha$ of $F-G$ satisfies $|a_\alpha|_v \le |p|_v^c$. Alternatively, for some $m\in\bN$ we use the notation $F\equiv G\pmod{\pi^m}$ if $F-G\in \pi^m\fo_v[[z_1,\dots, z_g]]$. More generally, for power series $\cF:=(F_1,\dots, F_g)$ and $\cG:=(G_1,\dots, G_g)$ we write $\cF\equiv \cG\pmod{p^c}$ if $F_i\equiv G_i\pmod{p^c}$ for each $i$; similarly, $\cF\equiv \cG\pmod{\pi^m}$ if $F_i\equiv G_i\pmod{\pi^m}$ for each $i$. Finally, for each $n\in\bN$ we denote by $\cF^n$ the composition of $\cF$ with itself $n$ times.

We use the following result in Section~\ref{Burnside section}.
\begin{proposition}
\label{Mike's bound}
Let $C\in \fo_v^g$, let $L\in\GL_g(\fo_v)$, and let $F_1,\dots, F_g\in \fo_v[[z_1,\dots,z_g]]$,  such that for $z:=(z_1,\dots,z_g)$ we have 
$$\cF(z):=(F_1,\dots, F_g)(z)\equiv C+Lz\pmod{\pi}.$$
Let $m=p^{1+r}\cdot  \#\GL_g(k_v)$ where $r$ is any nonnegative integer larger than $(\log(e)-\log(p-1))/\log(2)$. Then $\cF^m(z)\equiv z\pmod{p^c}$ for some $c>1/(p-1)$.
\end{proposition}

\begin{proof}
Let $s:=\#\GL_g(k_v)$; then it is immediate that $\cF^{s}(z)\equiv D+z\pmod{\pi}$ for some $D\in\fo_v^g$. So, $\cF^{ps}\equiv z\pmod{\pi}$. Hence we are left to show that if $\cF(z)\equiv z\pmod{\pi}$ and if $r$ is the least nonnegative integer greater than $(\log(e)-\log(p-1))/\log(2)$, then $\cF^{p^r}(z)\equiv z\pmod{p^c}$ for some $c>1/(p-1)$. Clearly, if $\cG(z)\equiv z\pmod{p^c}$ then also $\cG^{p^k}(z)\equiv z\pmod{p^c}$ for any positive integer $k$.

If $e<p-1$, then $r=0$ works since $|\pi|_v=|p|_v^{1/e}=p^{-1/e}<p^{-1/(p-1)}$. So, from now  on, we assume $e\ge p-1$.

We let $\cF(z)=z+\cH(z)$, where each coefficient of $\cH$ is in $\pi\cdot \fo_v$. Then $\cF^p(z)=z+p\cH(z)+\cH_1(z)$, where  $\cH_1\equiv 0\pmod{\pi^2}$.  Thus $\cF^p(z)\equiv z\pmod{\pi^2}$. 
By induction we obtain that  
$$\cF^{p^r}(z)\equiv z\pmod{\pi^{\min \{e+1,2^r\}}}.$$ 
So, if $r> (\log(e)-\log(p-1))/\log(2)$ then $|\pi|_v^{2^r}=p^{-\frac{2^r}{e}}<p^{-\frac{1}{p-1}}$, while $|\pi|_v^{e+1}<|p|_v\le p^{-\frac{1}{p-1}}$, and so indeed
$$\cF^{p^r}\equiv z\pmod{p^c}\text{ for some }c>\frac{1}{p-1},$$
which yields the desired conclusion.
\end{proof}

\subsection{Algebraic geometry}

We need the following application of the implicit function theorem on Banach spaces. 

\begin{proposition}
\label{dense}
Let $(K_v, |\cdot |_v)$ be a finite extension of $\Q_p$ with residue
field $k_v$, and let $\fo_v$ be the ring of $v$-adic integers of
$K_v$. Let $X$ be a quasiprojective variety defined over $K_v$, let
$\cX$ be a $\fo_v$-scheme whose generic fiber is isomorphic to $X$,
let $r: \cX(K_v) \lra \cXb(k_v)$ be the usual reduction map to the
special fiber $\cXb$ of $\cX$, and let $\iota: \cX(\fo_v) \lra X(K_v)$
be the usual map coming from base extension.  Let
$\alpha\in\cX(\fo_v)$ such that $\iota(\alpha)$ is a smooth point
on $X$ and let $U_{\bar{\alpha}}=\{\beta\in \cX(\fo_v)\text{ : }
r(\alpha) = r(\beta)\}$ and let $U = \iota( U_{\bar{\alpha}})$.
Then $U$ is Zariski dense in $X$.
\end{proposition}

\begin{proof}
  Let $x = \iota(\alpha)$.  We consider an affine chart containing the
  point $x\in X$ after viewing $X$ as a subset of the $n$-dimensional
  projective space defined over $K_v$. So, letting $d=\dim(X)$, then
  there exist $(n-d)$ polynomials $f_i$, which we may supposed are
  defined over $\fo_v$ in $n$ variables $z_1,\dots,z_n$ such that
  locally at $x$ the variety $X$ is the zero set of the polynomials
  $f_i$. Furthermore, since $x$ is a nonsingular point for $X$, the
  Jacobian matrix $\left(df_i/dz_j\right)_{i,j}$ has rank
  $n-d$. Without loss of generality we may assume the minor
  $\left(df_i/dz_j\right)_{1\le i,j\le n-d}$ is invertible.

  We let $x=(x_1,\dots,x_n)$ be the coordinates of the point $x$ in
  the above affine chart; each $x_i\in \fo_v$. Then $U_{\ba}$ is
  identified with points $(z_1, \dots, z_n) \in \fo_v^n$ such that
  $z_i \equiv x_i \pmod{\pi_v}$ for $\pi_v$ a generator for the
  maximal ideal in $\fo_v$. Using the Implicit Function Theorem (see
  \cite[Theorem 5.9, page 19]{Lang-differential}), we see that there
  exists a sufficiently small $p$-adic neighborhood $U_0$ of
  $(x_{n-d+1},\dots, x_n)$, there exists a $p$-adic neighborhood $V_0$
  of $(x_1,\dots,x_{n-d})$, and there exists a $p$-adic analytic
  function $g:U_0\lra V_0$ such that $g(x_{n-d+1},\dots,
  x_n)=(x_1,\dots, x_{n-d})$ and moreover for each $\gamma\in U_0$ we
  have $(g(\gamma),\gamma)\in X(K_v)$. Furthermore, at the expense of
  shrinking both $U_0$ and $V_0$ we may assume that for each
  $\gamma\in U_0$, the point $(g(\gamma),\gamma)$ is in $U$. Since
  $U_0 \subset \bA^d$ is a $d$-dimensional $K_v$-manifold we conclude
  that $U$ is Zariski
  dense in $X$.
\end{proof}

The following result is a consequence of Proposition~\ref{dense} for varieties defined over number fields. For each quasiprojective variety $X$ defined over a number field $K$, there exists a finite set $S$ of places (containing all archimedean places) and there exists a $\fo_{K,S}$-scheme $\cX$ whose generic fiber is isomorphic to $X$ (where $\fo_{K,S}$ is the subring containing all $u\in K$ such that $|u|_v\le 1$ for all $v\notin S$). In particular, we can prove the following result for $(\fo_K)_v$-schemes, where $(\fo_K)_v$ is the localization at the nonarchimedean place $v$ of the ring of algebraic integers $\fo_K$ of $K$. 

\begin{prop}
\label{dense number fields}
Let $K$ be a number field, let $v$ be a nonarchimedean place of $K$,
and let $(\fo_K)_v$ be the localization of $\fo_K$ at the place
$v$. Let $X$ be a quasiprojective variety defined over $K$, and let
$\cX$ be an $(\fo_K)_v$-scheme whose generic fiber is isomorphic to
$X$, let $r: \cX((\fo_K)_v) \lra \cXb(k_v)$ be the usual reduction
map, and let $\iota: \cX((\fo_K)_v) \lra X(K)$ be the usual map coming
from base extension.  Let
$\alpha\in\cX((\fo_K)_v)$ such that $\iota(\alpha)$ is a smooth point
on $X$, let  and let $U$ be the set of all $y \in
X(\Kbar)$ such that the Zariski closure of $y$ intersects $\cXb$ at $\alpha_v$.  Then $U$ is Zariski dense in $X$.    
\end{prop}
 
\begin{proof}
  Let $K_v$ be the completion of $K$ with respect to $|\cdot |_v$, and
  let $\fo_v$ be the ring of $v$-adic integers of $K_v$. Then using
  Proposition~\ref{dense} there exists a set $U_1 \subset
  \cX_{\fo_v}(\fo_v)$ whose intersection with the generic fiber
  $\cX_{K_v}$ is a  Zariski dense subset of $\cX_{K_v}$ (where $\cX_{\fo_v}$ is the base extension of $\cX$ to $\Spec(\fo_v)$, while $\cX_{K_v}$ is its generic fiber). We identify $U_1$
  with its intersection with the generic fiber $\cX_{K_v}$. Arguing as in the proof of Proposition~\ref{dense}
  we consider a system of coordinates for an affine subset $X_1\subset
  X$ containing $x$ (also defined over $K$), and find an open set
  $U_0\subset \bA^d(K_v)$ and a $v$-adic analytic function $g:U_0\lra
  K_v^{n-d}$ such that for each $z\in U_0$, we have $(g(z),z)\in
  X_1(K_v)\subset X(K_v)$.  Furthermore, for each such point
  $(g(\gamma),\gamma)\in X(K_v)$ there exists a section $\beta$ of
  $\cX_{\fo_v}$ whose intersection with the special fiber is $\alpha_v$, while
  its intersection with the generic fiber is $(g(\gamma),\gamma)$.

  We let $\pi:X_1\lra \bA^d$ be the projection on the last $d$
  coordinates.  Then using the Fiber Dimension Theorem \cite[Section
  6.3]{Shaf1} we conclude that there exists an open Zariski subset
  $U_2\subseteq \bA^d$ such that for each $\gamma\in U_2(\Kbar)\cap
  U_0$, the fiber $\pi^{-1}(\gamma)$ is a $\Kbar$-variety of dimension
  $0$ (here we use that $X$ and also $X_1$ are defined over
  $K$). Since $U_0\subset \bA^d$ is a $d$-dimensional
  $K_v$-manifold and $U_2$ is the complement in $\bA^d$ of a proper
  algebraic subvariety defined over $\Kbar$, we conclude that
  $U_2(\Kbar)\cap U_0$ is Zariski dense in $\bA^d$.  For each $\gamma
  \in U_2(\Kbar)\cap U_0$, we have 
  \[ (g(\gamma), \gamma) \in U_1 \cap \pi^{-1}(\gamma) \subset
  U_1 \cap X_1(\Kbar).\]  
  Let $h$ denote the map from $U_2(\Kbar)\cap U_0$ to $X_1$ sending
  $\gamma$ to $(g(\gamma), \gamma)$. Then the dimension of
  the closure of $U_2(\Kbar)\cap U_0$ is equal to the dimension of the
  closure of $h (U_2(\Kbar)\cap U_0)$ since $\pi \circ h$ is the
  identity on $U_2(\Kbar)\cap U_0$ and $\pi$ is finite-to-one on $h
  (U_2(\Kbar)\cap U_0)$.  Since this dimension is $d$, which is also
  the dimension of $X_1$, we see that $h (U_2(\Kbar)\cap U_0) \subseteq
  U$ is Zariski dense in $X_1$ and thus $U$ is Zariski dense in $X$.  
\end{proof}

%%%%%%%%%%%%%%%%%%%%%%%%%%%%%%%%%%%%%%%%%%%%%%%%%%%%%%%%%%%%%%%%%%%%%%%%%%%%%%%%
%%%%%%%%%%%%%%%%%%%%%%%%%%%%%%%%%%%%%%%%%%%%%%%%%%%%%%%%%%%%%%%%%%%%%%%%%%%%%%%%

\section{Burnside's problem}
\label{Burnside section}

In this section we continue with the notation from Section~\ref{v-adic analysis} for $g$, $p$, $(K_v,|\cdot |_v)$, $\fo_v$, $\pi$, $k_v$, $e$ and $r$. In addition, assume $p>2$.

Our first result gives an upper bound for the size of torsion of the automorphism group of a quasiprojective variety $X$ defined over a local field.  So, our setup is as follows: for a $\fo_v$-scheme $\cX$, we let $\cXb$ be its special fiber (over $k_v$). For a point $\alpha\in\cX(\fo_v)$, we let its residue class $\cU_\ba=\{\beta\in \cX(\fo_v)\text{ : }\overline{\beta}=\overline{\alpha}\}$, where $\overline{\gamma}\in\cXb(k_v)$ is the reduction modulo $v$ of $\gamma\in\cX(\fo_v)$. Finally, we note that if $\ba$ is a smooth point, then each $\beta\in\cU_\ba$ is also a smooth point.

\begin{theorem}
\label{the key result}
Let $\cX$ be a $\fo_v$-scheme whose generic
fiber is a $K$-variety of dimension $g$, and let $G \subseteq \Aut(\cX)$ be a torsion
group. If $\cX(\fo_v)$ contains a smooth point, then $G$ is finite and $\#G \le (\# k_v)^{g(1+e)\cdot \binom{g+e+1}{g}}\cdot\#\GL_g(k_v)\cdot  \#\overline{\cX}(k_v)$.   
\end{theorem}

\begin{proof}
We let $\alpha\in \cX(k_v)$ be a smooth point and let $G_\ba$ be the subgroup of $G$
consisting of all $\sigma$ such that $\sigma \ba = \ba$.  Since
$[G:G_\ba] \leq \#\overline{\cX}(k_v)$, it will suffice to bound the size of
$G_\ba$.  

Let $\cO_{\ba}$ denote the local ring of $\cX$ at $\ba$,   let
let ${\hat \fm_\ba}$ denote its maximal ideal, and let ${\hat
  \cO_\ba}$ denote the completion of $\cO_{\ba}$ at $\fm_\ba$.  Since $\alpha\in \cX$ is smooth,
the quotient ${\hat \cO_{\ba}} / (\pi)$ is also regular.  By the Cohen
structure theorem for regular local rings (see \cite[Theorem 9]{Cohen}
or \cite[Theorem 29.7]{Mats2}), the quotient ring ${\hat \cO_{\ba}} /
(\pi)$ can be written as formal power series $k_v [[ y_1,\dots, y_g
]]$.  Choosing $z_i \in {\hat \fm_v}$ for $i = 1, \dots, g$ such that
the residue class of each $z_i$ is equal to $y_i$, we obtain a minimal
basis $\{ \pi, z_1, \dots, z_g \}$ for ${\hat \fm_v}$ (see
\cite{Cohen}).  Thus, we see that ${\hat \cO_\ba}$ is naturally
isomorphic to a formal power series ring $\fo_v[[z_1, \dots, z_g]]$.

Arguing exactly as in the proof of \cite[Proposition~2.2]{BGT} we then
obtain that there is a $v$-adic analytic
isomorphism $\iota: \cU_\ba \lra \fo_v^g$, such that for any $\sigma
\in G_\ba$, there are power series $F_1, \dots, F_g \in \fo_v[[z_1,
\dots, z_g]]$ with the properties that
\begin{enumerate}
\item[(i)] each $F_i$ converges on $\fo_v^g$;
\item[(ii)] for all $(\beta_1, \dots, \beta_g) \in \fo_v^g$, we have
\begin{equation}\label{power eq}
\iota(\sigma( \iota^{-1} (\beta_1, \dots, \beta_g))) =   (F_1(\beta_1,\dots, \beta_g),\dots, F_g(\beta_1,
\dots, \beta_g)); \text { and }
\end{equation}
\item[(iii)] each $F_i$ is congruent to a linear polynomial mod $v$ (in
  other words, all the coefficients of terms of degree greater than one
  are in the maximal ideal $\fm_v$ of $\fo_v$). Moreover, for each $i$, we have
$$F_i(z_1,\dots,z_g)=\frac{1}{\pi}\cdot H_i(\pi z_1,\dots, \pi z_g),$$
for some $H_i\in \fo_v[[z_1,\dots, z_g]]$.  
\end{enumerate}

Denoting $\vb = (\beta_1, \dots, \beta_g)$ and $\iota \sigma
\iota^{-1}$ as $\cF_\sigma$, we thus have
\begin{equation}\label{p}
\cF_\sigma (\vb) \equiv C_\sigma+L_\sigma (\vb) \pmod{v}
\end{equation}
for a $C_\sigma\in\fo_v^g$ and a  $g \times g$ matrix $L_\sigma$ with coefficients in $\fo_v$.
Let $\bL_\sigma$ be the reduction of $L_\sigma$ modulo $\pi$.  Since
$\sigma$ is an \'etale map of $\fo_v$-schemes,  $\bL_\sigma$ must be
invertible.  We define $D_\ba: G_\ba \lra \bG_a^g(k_v)\rtimes\GL_g(k_v)$ by
$D_\ba(\sigma) = (\overline{C_\sigma},\bL_\sigma)$, where $\bG_a^g(k_v)\rtimes \GL_g(k_v)$ is the group of affine transformations of $k_v^g$.

We clearly have $\cF_{\sigma_1 \sigma_2} = \cF_{\sigma_1}
\cF_{\sigma_2}$ for $\sigma_1, \sigma_2 \in G_\ba$.  Reducing modulo
$\pi$, it follows from \eqref{p} that  $D_\ba(\sigma_1 \sigma_2) =
D_\ba(\sigma_1) D_\ba(\sigma_2)$.  Thus, $D_\ba$ is a group
homomorphism; let $G_{\ba, 1}$ be the kernel of $D_\ba$. 

Next we bound $\# G_{\ba,1}$. We consider the map 
$$E_\ba:G_{\ba,1}\lra \cV_g:=\left(\left(\fo_v/\pi^{e+1}\fo_v\right)[[z_1,\dots, z_g]]/(z_1,\dots, z_g)^{e+2}\right)^g,$$
given by reducing each component of $\cF_\sigma\in G_{\ba,1}$ modulo $\pi^{e+1}$.  Using property (iii) above we observe that $E_\ba$ is indeed well-defined and that it satisfies $E_\ba(\sigma_1\sigma_2)=E_\ba(\sigma_1)E_\ba(\sigma_2)$. Furthermore, because each $E_\ba(\sigma)$ for $\sigma\in G_{\ba,1}$ is an invertible power series, we conclude that $E_\ba$ restricts to a group homomorphism from $G_{\ba,1}$ into the subgroup of units $\cF$ of $\cV_g$ (with respect to the composition of functions) which satisfy the congruence $\cF(z)\equiv z\pmod{\pi}$. Since this group of units has at most $(\#k_v)^{eg\cdot \binom{g+e+1}{g}}$ elements, we are left to show that if $\sigma\in \ker E_\ba$, then $\sigma$ is the identity. 

Indeed, if $\cF_\sigma(z) \equiv z\pmod{\pi^{e+1}}$ for each $z\in\fo_v^g$, then $\cF_\sigma(\vb)\equiv \vb\pmod{p^c}$ for some $c>1/(p-1)$. Now fix $\vb$; by \cite[Theorem~1]{Poonen-p-adic}, there are $v$-adic
analytic power series $\theta_1, \dots, \theta_g \in \fo_v[z]$,
convergent on $\fo_v$, such that 
\[ \cF_\sigma^n(\vb) = (\theta_1(n), \dots, \theta_g(n)) \] for all $n
\in \bN$.  Since $\sigma$ has finite order, there is an $N_\sigma$ such that
$\cF_\sigma^{N_\sigma}$ is the identity, we so have $\theta_i (k N_\sigma) = \beta_i$
for all $k \in \bN$.  Hence $\theta_i(u) - \beta_i$ has infinitely
many zeros $u\in\fo_v$.  Therefore, $\theta_i(u) - \beta_i$ is
identically zero since any nonzero convergent power series on $\fo_v$
has finitely many zeros in $\fo_v$.  Thus, $\cF_\sigma(\vb) = \vb$ for
all $\vb \in \fo_v^g$.

  Hence, we
have $\sigma(z) = z$ for all $z \in \cU_\ba$.  Since
$\cU_\ba$ is Zariski dense in $\cX$ (they are both $g$-dimensional $K_v$-manifolds), we have that $\sigma$ acts on
identically on all of $X$.  This concludes our proof.
\end{proof}

The following result is an immediate corollary of Theorem~\ref{the key result} since each torsion point of a semiabelian variety $X$ induces a torsion element of $\Aut(X)$.
\begin{cor}
\label{torsion of semiabelian}
Let $X$ be a semiabelian $\fo_v$-scheme whose generic fiber has dimension $g$. Then $\#\cX_{\tor}(\fo_v)\le (\# k_v)^{g(1+e)\cdot \binom{g+e+1}{g}}\cdot\#\GL_g(k_v)\cdot  \#\overline{\cX}(k_v)$.
\end{cor}

If $G$ is cyclic, then we can give a much better bound for $\# G$. In fact, Theorem~\ref{key result for subvarieties} yields an upper bound for the order of any $\fo_v$-subscheme $\cY$ of $\cX$ which is preperiodic under the action of an \'etale endomorphism $\Phi$ of $\cX$; this gives a higher dimensional generalization of the main result of Hutz \cite{Hutz-0}. We recall that $r$ is the smallest nonnegative integer larger than $(\log(e)-\log(p-1))/\log(2)$, where $e$ is the ramification index of $K_v/\Q_p$. 

\begin{proof}[Proof of Theorem~\ref{key result for subvarieties}.]
  We use the same setup as in the proof of Theorem~\ref{the key
    result}.  Let $\beta \in \cY(\fo_v)$ be a smooth point on $\cY$.  Since $\cXb(k_v)$ is
  finite, there is an $\ell \geq 0$ such that the residue class of
  $\Phi^\ell(\beta)$ is periodic under $\Phi$; we note this residue
  class as $U_0$ and we denote $\Phi^\ell(\cY)$ as $\cY '$.  There is
  then an integer $k$ such that $\Phi^k(U_0) = U_0$ and $k + \ell \leq
  \#\cXb(k_v)$.

  Since $\Phi^\ell(\beta)\in \cY'(\fo_v)\cap U_0$ is a smooth point on the generic fiber of $\cY'$, Proposition~\ref{dense} yields that $\cY'(\fo_v)\cap U_0$ is  
  Zariski dense in $\cY '$.   Let $x\in U_0\cap \cY '(\fo_v)$, let $m:=p^{1+r}\cdot \#\GL_g(k_v)$, and let $\Psi:=\Phi^{mk}$. Arguing as in the proof of Theorem~\ref{the key result} (note that in order to apply the strategy from \cite[Proposition~2.2]{BGT} we require that $\Phi$ is \'etale and that $x$ is smooth on $\cX$ only), and also applying Proposition~\ref{Mike's bound}, we obtain that $\cF_\Psi(z)\equiv z\pmod{p^c}$ for some $c>1/(p-1)$.  Hence, by \cite[Theorem~1]{Poonen-p-adic}, there exists a $v$-adic analytic function $\cG_{\Psi, x}:\fo_v\lra U_0$ such that $\cG_{\Psi,x}(n)= \Psi^n(x)$.

  Now, let $F$ be a polynomial in the vanishing ideal of
  $\cY '$. Because $\cY '$ is periodic, there exists a positive integer
  $N$ such that $\Phi^N(\cY ')=\cY '$, and thus $F(\Phi^{nN}(x))=0$ for
  each $n\in\bN$. On the other hand, $\cG_{\Psi,x}(n)=\Phi^{nmk}(x)$  and
  so, $F(\cG_{\Psi,x}(nN))=0$ for all $n\in\bN$. Since a nonzero $v$-adic
  analytic function cannot have infinitely many zeros in $\mathbb{N}\subset
  \fo_v$, we conclude that $F(\cG_{\Psi,x}(n))=0$ for all $n\in\mathbb{N}$; in particular, $F(\Phi^{mk}(x))=0$. Thus, $\Phi^{mk}(x)\in \cY '$,  and so 
  $\Phi^{km} (\cY') = \cY'$.  Since $k + \ell \leq \#\cXb(k_v)$, we
  have that the length of the orbit of $\cY$ under $\Phi$ is bounded by $km + \ell \leq  m\cdot \#\cXb(k_v)=p^{1+r}\cdot \#\GL_g(k_v)\cdot \#\cXb(k_v)$.
\end{proof}

The following two results are simple consequences of Theorem~\ref{key result for subvarieties}.

\begin{cor}
\label{cyclic key result}
Let $\cX$ be a $\fo_v$-scheme whose generic
fiber $X$ has dimension $g$, let $\Phi: \cX \lra \cX$ be \'etale, and
let $\alpha\in \cX(\fo_v)$ be a smooth preperiodic point. Then the length of its orbit is bounded
by $p^{1+r}\cdot \#\GL_g(k_v)\cdot \#\cXb(k_v)$.
\end{cor}

\begin{cor}
\label{cyclic torsion of semiabelian}
Let $X$ be a semiabelian $\fo_v$-scheme whose generic fiber has dimension $g$. Then each torsion point of $\cX(\fo_v)$ has order bounded above by $p^{1+r}\cdot\#\GL_g(k_v)\cdot  \#\overline{\cX}(k_v)$.
\end{cor}

Our arguments above allow us to  show that for any  field $K$ of characteristic $0$, and for any finitely generated extension $L/K$, then each  finitely generated torsion subgroup of $\Aut(L)$ fixing $K$ is finite, i.e., Burnside's problem has a positive answer. At the expense of replacing $L$ by a finite extension and then viewing $L$ as the function field of a geometrically irreducible quasiprojective variety defined over $K$, we obtain the geometric formulation of the Burnside problem from Theorem~\ref{Burnside p}.

\begin{proof}[Proof of Theorem~\ref{Burnside p}.]
Let $\sigma_1,\dots, \sigma_m$ be a finite set of generators for $H$, and let $K$ be a finitely generated field such that $X$, $\sigma_1,\dots, \sigma_m$ are all defined over $K$. After passing to a finite extension of the base, we may assume that $X(K)$ contains a smooth point $\alpha$.    Let $R$ be a finitely generated $\bZ$-algebra containing all the
coefficients of all the polynomials defining $X$ in some projective
space, along with all the coefficients of all the polynomials defining all the
$\sigma_i$ locally, as in the proof of \cite[Theorem 4.1]{BGT}.   By
\cite[Proposition 4.3]{BGT}, since a finite intersection of dense open
subsets is dense, we see that there is a dense open subset $U$ of
$\Spec R$ such that:
\begin{enumerate}
\item there is a scheme $\cX_{U}$ that is quasiprojective over $U$, and whose generic fiber equals $X$;
\item each fiber of $\cX_{U}$ is geometrically irreducible;
\item each $\sigma_i$ extends to an automorphism ${\sigma_i}_{U}$ of $\cX_{U}$; and
\item $\alpha$ extends to a smooth section $U \lra \cX_U$.
\end{enumerate}
Now, arguing as in \cite[Proposition 4.4]{BGT}, and using \cite[Lemma
3.1]{Bell}, we see that there is an embedding of $R$ into $\bZ_p$ (for some prime $p\ge 5$), and a
$\bZ_p$-scheme $\cX_{\bZ_p}$ such that
\begin{enumerate}
\item $\cX_{\bZ_p}$ is quasiprojective over $\bZ_p$, and its
  generic fiber equals $X$;
\item both the generic and the special fiber of $\cX_{\bZ_p}$ are geometrically irreducible;
\item each $\sigma_i$ extends to an automorphism $(\sigma_i)_{\bZ_p}$ of $\cX_{\bZ_p}$; and
\item $\alpha$ extends to a smooth section $ \Spec \bZ_p \lra \cX_{\bZ_p}$.
\end{enumerate}
Then Theorem~\ref{the key result} finishes our proof. 
\end{proof}

%#######################################################################
%#######################################################################
%#######################################################################

%%%%%%%%%%%%%%%%%%%%%%%%%%%%%%%%%%%%%%%%%%%%%%%%%%%%%%%%%%%%%%%%%%%%%%%%%%%%%%%%
%%%%%%%%%%%%%%%%%%%%%%%%%%%%%%%%%%%%%%%%%%%%%%%%%%%%%%%%%%%%%%%%%%%%%%%%%%%%%%%%

%%%%%%%%%%%%%%%%%%%%%%%%%%%%%%%%%%%%%%%%%%%%%%%%%%%%%%%%%%%%%%%%%%%%%%%%%%%%%%%%
%%%%%%%%%%%%%%%%%%%%%%%%%%%%%%%%%%%%%%%%%%%%%%%%%%%%%%%%%%%%%%%%%%%%%%%%%%%%%%%%

\section{Bounds on the number of periodic Hypersurfaces}
\label{bounds for hypersurfaces}

In this section we give explicit bounds on the number of $\sigma$-periodic hypersurfaces when $\sigma$ is an automorphism of an irreducible quasi-projective variety $X$ which preserves no rational fibration.  In particular, we show the number of $\sigma$-periodic hypersurfaces is finite unless there exists a nonconstant rational function $f$ such that $f\circ \sigma = f$.  Moreover, we are able to give a bound for both the lengths of periods and the number of $\sigma$-periodic hypersurfaces in terms of geometric data, although this bound depends upon the field of definition for $\sigma$. We note that Cantat \cite[Theorem B]{Cantat_preperiodic} proved there exists a bound $N(\sigma)$ (depending on $\sigma$) such that if there exist more than $N(\sigma)$ irreducible periodic hypersurfaces, then $\sigma$ must preserve a nonconstant rational fibration. In the case that $\sigma$ is defined over a number field $K$ and there is a point $x\in X(K)$ with a dense orbit under $\sigma$, we are able to give a bound that depends only upon the dimension of $X$ and the Picard number of a projective closure (see Theorem~\ref{Vojta consequence}).  We begin with a lemma about ranks of multiplicative subgroups of a field that are stable under an automorphism of the field. As a matter of notation, for an automorphism $\sigma$ of a field $K$, we denote by $K^\sigma$ the set of all fixed points of $\sigma$.

\begin{prop} Let $k$ be an algebraically closed of characteristic zero and let $K$ be a finitely generated field extension of $k$.  Suppose that $\sigma:K\to K$ is a $k$-algebra automorphism with $K^{\sigma}=k$.  If $G$ is a finitely generated $\sigma$-invariant subgroup of $K^*$ then the rank of $G/(G\cap k^*)$ is at most ${\rm trdeg}_k(K)$.
\label{lem: bound1}
\end{prop}
\begin{proof}  Suppose, towards a contradiction, that the rank of $G/(G\cap k^*)$ is $m>{\rm trdeg}_k(K)$ and suppose that $x_1,\ldots ,x_m$ are elements of $G$ whose images in $G/(G\cap k^*)$ generate a free abelian group of rank $m$.  Then there is some nonzero polynomial 
$P(t_1,\ldots ,t_m)\in k[t_1^{\pm 1},\ldots ,t_m^{\pm 1}]$ such that $P(x_1,\ldots ,x_m)=0$.  We write $P$ as 
$$\sum_{j_1,\ldots ,j_m} c_{j_1,\ldots ,j_m} t_1^{j_1}\cdots t_m^{j_m}$$ and we let
$$N:=\#\{(j_1,\ldots ,j_m)\colon c_{j_1,\ldots ,j_m} \neq 0\}.$$
We may take $P$ so that $N>1$ is minimal.   By multiplying $P$ by an appropriate monomial and nonzero constant, we may also assume that the constant coefficient of $P$ is equal to one.  Then we have
$$P(\sigma^i(x_1),\ldots ,\sigma^i(x_m))=0$$ for all $i\in \mathbb{Z}$.  In other words, for each integer $i$, 
$$(z_{j_1,\ldots ,j_m})_{(j_1,\ldots ,j_m)}=(\sigma^i(x_1^{j_1}\cdots x_m^{j_m}))_{(j_1,\ldots ,j_m)}\in G^N$$ is a solution to the $S$-unit equation 
$$\sum_{j_1,\ldots ,j_m} c_{j_1,\ldots ,j_m} z_{j_1,\ldots ,j_m}=0.$$  By minimality of $N$, each of these solutions is primitive; that is, no proper subsum vanishes.  (If some proper non-trivial subsum of 
$$\sum_{j_1,\ldots ,j_m} c_{j_1,\ldots ,j_m} \sigma^i(x_1)^{j_1}\cdots \sigma^i(x_m)^{j_m}$$ vanished for some $i$, then we could apply $\sigma^{-i}$ to this subsum and get a smaller $N$, contradicting minimality.)
By the theory of $S$-unit equations for fields of characteristic zero (see Evertse et al. \cite{ESS}), we know there are only finitely many primitive solutions in $G^N$ to the equation
$$\sum_{j_1,\ldots ,j_m} c_{j_1,\ldots ,j_m} z_{j_1,\ldots ,j_m}=0$$ up to multiplication by elements of $G$.  
It follows that there is some $M>0$ and some $y\in G$ such that $$\sigma^M(x_1^{j_1}\cdots x_m^{j_m})=y x_1^{j_1}\cdots x_m^{j_m}$$ whenever $c_{j_1,\ldots ,j_m} \neq 0$.  Since $c_{0,\ldots ,0}\neq 0$, we see that $y=1$.  Thus if we pick $(j_1,\ldots ,j_m)\neq (0,\ldots , 0)$ with $c_{j_1,\ldots ,j_m}\neq 0$ then $\sigma^M$ fixes $x_1^{j_1}\cdots x_m^{j_m}$, which by assumption is not in $k^*$, and so $\sigma^M$ has a fixed field of transcendence degree at least one over $k$.  Since the fixed field of $\sigma^M$ is a finite extension of the fixed field of $\sigma$, we see that the fixed field of $\sigma$ has transcendence degree at least one over $k$, a contradiction.  The result follows.
\end{proof}

As a corollary, we obtain the following result.

\begin{thm} Let $K$ be a finitely generated extension of $\mathbb{Q}$ and let $X$ be an irreducible quasi-projective variety defined over $K$.  Then there exists a positive constant $N=N(X,K)$ such that whenever $\sigma\in {\rm Aut}_K(X)$ has the property that there are no nonconstant $f\in \overline{K}(X)$ with $f\circ \sigma=f$ there are at most $N$ $\sigma$-periodic hypersurfaces and they all have period at most $N$.  Moreover, $N$ can be taken to be ${\rm rank}({\rm Cl}(\tilde{X}))+{\rm dim}(X)$, where $\tilde{X}$ is the normalization of $X$.
\label{thm: bound}\end{thm}
 We note that when $Y$ is a normal quasi-projective variety over a finitely generated extension of $\mathbb{Q}$, we have ${\rm Cl}(Y)$ has finite rank \cite[Lemma 5.6 (1)]{BRS}.  We will find it convenient to regard $K$ as a subfield of $\mathbb{C}$ throughout. 
 
 \begin{proof}[Proof of Theorem \ref{thm: bound}]  It is no loss of generality to assume that $X$ is normal. 
 Suppose that there is no nonconstant $f\in \overline{K}(X)$ with $f\circ\sigma=f$.  
Let $N:={\rm rank}({\rm Cl}(X))+{\rm dim}(X)$, and suppose that we have $N+1$ distinct $\sigma$-periodic hypersurfaces $Y_0,\ldots ,Y_N$.  By replacing $\sigma$ by an iterate, we may assume that $\sigma(Y_i)=Y_i$ for all $i$.  By relabeling if necessary, we may assume that there is some $m\le N-{\rm dim}(X)-1$ such that $[Y_0],\ldots ,[Y_m]$ generate a free $\mathbb{Z}$-module of ${\rm Cl}(Y)$ and that for $i>m$, $[Y_0],\ldots ,[Y_m],[Y_i]$ are dependent in ${\rm Cl}(X)$.  
 This means that for $i\in \{N-{\rm dim}(X),\ldots ,N\}$, there is a principal divisor $(f_i)=c_{i,i} [Y_i]
 +\sum_{j=0}^m c_{i,j} [Y_j]$, where the $c_{i,j}$ are integers and $c_{i,i}$ is nonzero.  By construction, we have $f_i\circ \sigma$ has the same divisor as $f_i$ for $i=N-{\rm dim}(X),\ldots ,N$.  Also, the $f_i$ generate a free abelian subgroup of $\mathbb{C}(X)^*$, which can be seen by noting that the valuation on $\mathbb{C}(X)$ induced by $Y_i$, $\nu_{Y_i}$, has the property that $\nu_{Y_i}(f_i)$ is nonzero but $\nu_{Y_j}(f_i)=0$ for $j\in  \{N-{\rm dim}(X),\ldots ,N\}\setminus \{i\}$.  
 
 Since $f_i\circ \sigma$ has the same divisor as $f_i$, we see that $f_i\circ \sigma/f_i$ is in $\Gamma(X,\mathcal{O}_X)^*$.  Let $G$ denote the subgroup of $\mathbb{C}(X)^*/\mathbb{C}^*$ generated by $\Gamma(X,\mathcal{O}_X)^*/\mathbb{C}^*$ and by the images of the $f_i$.  Then we have shown that the rank of $G$ is at least ${\rm dim}(X)+1$.  Moreover, $G$ is finitely generated since $\Gamma(X,\mathcal{O}_X)^*/\mathbb{C}^*$ is finitely generated \cite[Lemma 5.6 (2)]{BRS}.  Furthermore, $\sigma$ induces an automorphism of $G$ since $\Gamma(X,\mathcal{O}_X)^*/\mathbb{C}^*$ is closed under application of $\sigma$ and since $f_i\circ \sigma \in \Gamma(X,\mathcal{O}_X)^*f_i$.  We now let $g_1,\ldots ,g_s$ be elements of $\mathbb{C}(X)^*$ whose images in $\mathbb{C}(X)^*/\mathbb{C}^*$ generate $G$.  Let $G_0$ denote the subgroup of $\mathbb{C}(X)^*$ generated by $g_1,\ldots ,g_s$.  Then there exist complex numbers $\lambda_1,\ldots ,\lambda_s$ such that $g_i\circ \sigma \in \lambda_i G_0$.  Let $H$ denote the subgroup of $\mathbb{C}(X)^*$ generated by $G_0$ and by $\lambda_1,\ldots ,\lambda_s$.  Then $H$ is finitely generated and by construction we have $h\circ \sigma\in H$ for all $h\in H$.  Furthermore, the rank of $H/(H\cap \mathbb{C}^*)$ is at least ${\rm dim}(X)+1$, since its rank is at least as large as the rank of $G$.  Lemma \ref{lem: bound1} gives a contradiction.  The result follows.
 \end{proof}
We note that for any complex variety $X$ with automorphism $\sigma$, there is some finitely generated extension $K$ of $\mathbb{Q}$ such that $X$ is defined over $K$ and such that $\sigma \in {\rm Aut}_K(X)$ and so Theorem \ref{thm: bound} can be applied using the value of $N(X,K)$ given in the statement of the theorem.

Also as a corollary of Theorem~\ref{thm: bound} we can prove that for any quasiprojective variety $X$ defined over $\Qbar$ under the action of an automorphism $\Phi$ which does not preserve a rational fibration,  there exist non-periodic codimension-$1$ subvarieties (defined over $\Qbar$). Indeed,  using Theorem~\ref{thm: bound}, there exist finitely many codimension-$1$ periodic subvarieties $Y_i$; in addition, let $N_1\in\mathbb{N}$ such that each $Y_i$ is fixed by $\Phi^{N_1}$. So we can find an algebraic point $x\in X(\Qbar)$ which is not contained in the above finitely many codimension-$1$ subvarieties $Y_i$. Then we simply take $Y$ be the intersection of $X$ (inside some projective space) with a hyperplane (defined over $\Qbar$) passing through $x$, but not containing $\Phi^{N_1}(x)$; then $Y$ is not periodic (since if it were, then it would be fixed by $\Phi^{N_1}$ but on the other hand, $\Phi^{N_1}(x)\notin Y(\Qbar)$), and therefore its orbit under $\Phi$ is Zariski dense in $X$.

%%%%%%%%%%%%%%%%%%%%%%%%%%%%%%%%%%%%%%%%%%%%%%%%%%%%%%%%%%%%%%%%%%%%%%%%%%%%%%%
%%%%%%%%%%%%%%%%%%%%%%%%%%%%%%%%%%%%%%%%%%%%%%%%%%%%%%%%%%%%%%%%%%%%%%%%%%%%%%%%

\section{Subvarieties with Zariski dense orbits}
\label{Zhang surfaces}

The setup for this Section is as follows: $X$ is a quasiprojective variety defined over $\C$, and $\Phi$ is an automorphism of $X$ which preserves no nonconstant rational fibration. Our goal is to prove Theorem~\ref{Zhang's conjecture for automorphisms}; we use Theorems~\ref{key result for subvarieties} and \ref{thm: bound}.

\begin{proof}[Proof of Theorem~\ref{Zhang's conjecture for automorphisms}.]
Arguing as before, for a suitable prime $p\ge 5$, we find a  $\Zp$-scheme $\cX$ such that 
\begin{enumerate}
\item[(i)] $X$ is the generic fiber of $\cX$, while the special fiber $\cXb$ of $\cX$ is a geometrically irreducible quasiprojective variety.
\item[(ii)] $\Phi$ extends to an automorphism of $\cX$.
\item[(iii)] there exists $x_0\in\cX(\Zp)$ such that its reduction $\overline{x_0}$ modulo $p$ is a smooth point of $\cXb$. 
\end{enumerate}  
Let $U_0:=\{x\in\cX(\Zp)\text{ : } \overline{x}=\overline{x_0}\}$ be the residue class of $x_0$ (since $\overline{x_0}$ is a smooth point on $\cXb$, then each $x\in U_0$ is also smooth on $\cX$). Furthermore, we identify each section in $U_0$ with its intersection with the generic fiber $X$. Using Theorem~\ref{key result for subvarieties}, there exists a positive integer $N_1$ such that each periodic subvariety $Y$ which contains a point from $U_0$ which is smooth also on $Y$ has period bounded above by $N_1$.

By Theorem~\ref{thm: bound} there exist at most finitely many
codimension-$1$ subvarieties which are fixed by $\Phi^{N_1}$. Let
$Y_1$ be the union of all these codimension-$1$ subvarieties. On the
other hand, by the definition of $N_1$, if $x\in U_0$ is (pre)periodic
then $\Phi^{N_1}(x)=x$. Because $\Phi$ has infinite order (since it
preserves no nonconstant rational fibration), the vanishing locus for
the equation $\Phi^{N_1}(x)=x$ is a proper subvariety $Y_0$ of
$\cX$. In conclusion, $Y_0\cup Y_1$ is a proper subvariety of $X$ and
therefore, there exists a Zariski dense set of  points $x\in
U_0\setminus (Y\cup Y_1)(\Z_p)$ (because $U_0$ is a $p$-adic manifold
of dimension larger than $\dim(Y_0\cup Y_1)$). Furthermore, we can choose
$x\in X(\Qbar)$ by Proposition~\ref{dense number fields}; finally note that $x$ is smooth since it is in $U_0$.

For each such point $x\in X(\Qbar)\cap U_0$ which is not contained in $Y_0\cup Y_1$, we can find a codimension-$2$ subvariety $Y$ (defined over $\Qbar$) whose orbit under $\Phi$ is Zariski dense in $X$. Indeed, we consider $X$ embedded into a large projective space $\bP^m$ and then intersect $X$ with two (generic) hyperplane sections $H_1$ and $H_2$ (defined over $\Qbar$) which pass through $x$, but not through $\Phi^{N_1}(x)$ (note that $\Phi^{N_1}(x)\ne x$ because $x\notin Y_0$). Furthermore, since $H_1$ and $H_2$ are generic sections passing through $x$, then $Y:=X\cap H_1\cap H_2$ is a codimension-$2$ irreducible  subvariety defined over $\Qbar$ and moreover $x\in Y$ is a smooth point. We claim that $Y$ is not periodic under $\Phi$. Otherwise since $Y$ intersects $U_0$ then it must be fixed by $\Phi^{N_1}$ (by Theorem~\ref{key result for subvarieties} and our choice for $N_1$). However, $x\in Y$, but $\Phi^{N_1}(x)\notin Y$, which shows that $Y$ is not fixed by $\Phi^{N_1}$, and thus $Y$ is not periodic under the action of $\Phi$. Let $Z$ be the Zariski closure of the orbit of $Y$ under the action of $\Phi$. Since $Y$ is not periodic under $\Phi$, then $\dim(Z)>\dim(Y)$. Now, if $\dim(Z)<\dim(X)$, then $Z$ is a codimension-$1$ subvariety, and in addition it is fixed by $\Phi^{N_1}$. Then it has to be contained in $Y_1$. However this is impossible since $x\in Z$ but $x\notin Y_1$. In conclusion, $Z=X$, as desired.
\end{proof}

In particular, if the codimension-$2$ subvariety $Y$ from the conclusion of Theorem~\ref{Zhang's conjecture for automorphisms} has the property that $Y(L)$ is Zariski dense in $Y$ (for some number field $L$ containing the field of definition for $\Phi$), then $X(L)$ is Zariski dense in $X$. So our Theorem~\ref{Zhang's conjecture for automorphisms} may be used to prove that certain varieties $X$ have a Zariski dense set of rational points, by reducing the problem to finding a potentially dense set of rational points on a codimension-$2$ subvariety $Y$ of $X$.

In the case that $\sigma: X\to X$ is defined over a number field $K$ and there is a point $x\in X(K)$ with dense orbit under $\sigma$, we obtain a much stronger upper  bound (that has no dependence on the number field) for the period of codimension-$1$ subvarieties of $X$ periodic under the automorphism.

\begin{proof}[Proof of Theorem~\ref{Vojta consequence}.]
  We extend $\sigma$ to a map $\cX' \lra \cX'$ where $\cX'$ is defined
  over the ring of integers $\fo_K$.  Let $R$ be the localization of
  $\fo_K$ away from all at the primes of bad reduction.  Then we obtain
  an automorphism of $R$-schemes $\sigma_0: \cX \lra \cX$.  Now, let
  $\cY$ be some projective closure for $\cX$; then $x$ meets $\cY
  \setminus \cX$ over at most finitely many finite primes, call this
  set $T$, and let $R'$ denote the localization of $R$ away from $T$.
  Let $Y$ be the generic fiber of $\cY$.

Let $W$ be an invariant subvariety of $X$.  Suppose that $W$ has at
least  $\dim X - 
  h^1(Y,\cO_Y) + \rho + 1$ geometric components, where $\rho$ is the
Picard number of $Y$ (the rank of its N\'eron-Severi group).  Then
clearly $x$ is not in $W$ so there is an at most finite set $T'$ of
primes at which $x$ meets $W$.   Let $S = T \cup T' \cup (\Spec \fo_K
\setminus \Spec R)$.  Then $x$ is $S$-integral relative to $W$ and,
since $\sigma^{-1}(W) = W$, we see that $\sigma^n(x)$ is $S$-integral relative
to $W$ for all $n$ (if $\sigma^n(x)$ met $W$ modulo a prime, then $x$
would meet $\sigma^{-n}(W)$ modulo that same prime).  But by a result of Vojta
\cite[Cor. 0.3]{V1}, this would mean that the orbit of $x$ was not
dense, since $W$ has at
least  $\dim X - 
  h^1(Y,\cO_Y) + \rho + 1$ geometric components, which gives
a contradiction.
\end{proof}

%%%%%%%%%%%%%%%%%%%%%%%%%%%%%%%%%%%%%%%%%%%%%%%%%%%%%%%%%%%%%%%%%%%%%%%%%%%%%%%
%%%%%%%%%%%%%%%%%%%%%%%%%%%%%%%%%%%%%%%%%%%%%%%%%%%%%%%%%%%%%%%%%%%%%%%%%%%%%%

\section{Other questions}
\label{other questions}

Poonen \cite{Poonen-uniform} has proposed a variant of Morton-Silverman's uniform
boundendess conjecture, where the morphisms vary across a general
families of self-maps of varieties rather than just the universal
family of degree-$d$ self-maps $\bP^N \lra \bP^N$.  In Poonen's
set-up, some fibers may have infinitely many preperiodic points.
Although that cannot happen in the case of preperiodic points of
morphisms $\bP^n \lra \bP^n$ (because of Northcott's theorem),
morphism $\bP^n \lra \bP^n$ {\it can} have infinitely many positive
dimensional periodic subvarieties.  For example, if $f$ is a
homogeneous two-variable polynomial of degree $n$, then the morphism
$\bP^2 \lra \bP^2$ given by $[x:y:z] \mapsto [f(x,z): f(y,z): z^n]$
has infinitely many $f$-invariant curves of the form $[x z^{n^k-1}:
f^k(x,z): z^{n^k}]$, where $f^k$ is the homogenized $k$-th iterate of the dehomogenized one-variable polynomial $x\mapsto f(x,1)$.  On the other hand, it is possible that one may be able
to bound the periods of the $f$-periodic subvarieties in general.

To state our question, we will need a little terminology.  To be
clear, we will say that $V$ is a $K$-subvariety of $X$ if $V$ is a
geometrically irreducible subvariety of $X$ defined over $K$.   Since
so little is known about this question, we will ask it in slightly
less generality than Poonen uses.    Given a morphism $\Phi:X \lra X$ and a $K$-subvariety $V$
of $X$ such that $V$ is periodic under the action of $\Phi$, we define
$\Per_{\Phi}(V)$ to be the smallest $n$ such that $\Phi^n(V) \subseteq V$.  

\begin{question}\label{uq1} 
Let $\pi: \cF \lra S$ be a morphism of varieties defined over a number
field $K$ and let $\Phi: \cF \lra \cF$ be an $S$-morphism.  For $s \in
S(K)$, we let $\cF_s$ be the fiber $\phi^{-1}(s)$ and let $\Phi_s$ be
the restriction of $\Phi$ to $\cF_s$.  Is there a constant $N_\cF$
such that for any $s \in S(K)$ and any periodic $K$-subvariety $V$ of
$\cF_s$, we have $\Per_{\Phi_s}(V) \leq N_\cF$?  

\end{question}

Even in the case where one can assign canonical heights to
subvarieties of $X$, there may be subvarieties of $X$ of positive
dimension having canonical height 0 that are not preperiodic
(see \cite{MM}).  Thus, we do not even know the answer to Question
\ref{uq1} even in the case of a constant family of maps.

\begin{question}\label{uq2}
  Let $\Phi: X \lra X$ be a morphism of varieties defined over a number
  field $K$.  Is there a constant $N_X$ such that for any periodic
  $K$-subvariety $V$ of $X$, we have $\Per_\Phi(V) \leq N_X$?   
\end{question}

We may also ask an analog of Question \ref{uq1} for finite subgroups
of automorphism groups.   

\begin{question}\label{bq1}
Let $\pi: \cF \lra S$ be a morphism of varieties defined over a number
field $K$.  For $s \in S(K)$, we let $\cF_s$ denote the fiber
$\pi^{-1}(s)$.  
Must the set
\[ \{ n \; | \; \text{ there is an $s \in S(K)$ such that $\Aut(\cF_s)$ has a subgroup of order $n$}
\} \] 
be finite?
\end{question}

The theorems of Mazur \cite{Mazur} and Merel \cite{Merel} show that
Questions \ref{uq1} and  \ref{bq1} has a
positive answer when $\cX$ is a family of elliptic curves.  Similarly,
work of Kond{\=o} \cite{Kondo} shows that Question \ref{bq1} has a
positive answer when $\cF$ is a
family of K3 surfaces.  

As with Question \ref{uq1}, we do not know the answer to Question
\ref{bq1} even in the constant family case.  On the other hand,
the bound in Theorem \ref{the key result} depends only on the
dimension of $X$ and the number of points in the special fiber of $X$
at the place $v$; by the Weil bounds of Deligne \cite{Weil1, Weil2},
the number of points on this special fiber can be bounded in terms of
$\#k_v$, the dimension of $X$, and the Betti numbers of $X$.  Thus,
one might expect that there is a bound on the the largest finite
subgroup of $\Aut(\cF_s)$ having good reduction at $v$ as $\cF_s$
varies in a family.

\bibliographystyle{amsalpha}
\bibliography{Drinbib}

\end{document}